\documentclass[12pt]{article}

\usepackage{fullpage,amsfonts,amsmath,amsthm,amssymb,mathrsfs,graphicx,upgreek}
\usepackage{enumitem}

\usepackage{bm}
\usepackage{bbm}
\usepackage{xcolor}
\usepackage{xcolor}
\usepackage[normalem]{ulem}
\usepackage{verbatim,url}
\usepackage{graphicx,hyperref,enumitem}
\usepackage[noadjust,sort]{cite}

\usepackage[a4paper, top=2.54cm, bottom=2.54cm, left=2.54cm, right=2.54 cm]{geometry}

\newtheorem{theorem}{Theorem}[section]
\newtheorem{corollary}[theorem]{Corollary}
\newtheorem{conjecture}[theorem]{Conjecture}
\newtheorem{lemma}[theorem]{Lemma}

\theoremstyle{definition}
\newtheorem{remark}[theorem]{Remark}

\newtheorem{prob}[theorem]{Problem}

\newcommand{\tr}{{\rm{tr}}}

\let\originalleft\left
\let\originalright\right
\renewcommand{\left}{\mathopen{}\mathclose\bgroup\originalleft}
\renewcommand{\right}{\aftergroup\egroup\originalright}

\newcommand{\expec}{\mathbb{E}}
\newcommand{\Exp}[1]{\expec[#1]}

\newcommand{\CovA}[1]{{\mbox{Cov}\left(#1\right)}}

\newcommand{\drat}{\lambda}

\def\G{{\boldsymbol{G}}}

\def\Pr{\operatorname{\mathbb P}}

\def\GoodS{S_{\rm good}}
\def\GoodT{T_{\rm good}}

\def\abs#1{\lvert#1\rvert} 
\def\Abs#1{\bigl\lvert#1\bigr\rvert}

\def\dfrac#1#2{\lower0.15ex\hbox{\large$\textstyle\frac{#1}{#2}$}}
\def\Dfrac#1#2{\raise0.05ex\hbox{\small$\displaystyle\frac{#1}{#2}$}}
\def\({\bigl(}
\def\){\bigr)}
\def\st{\mathrel{:}}

\def\tr{\operatorname{tr}}
\def\nicebreak{\vskip 0pt plus 50pt\penalty-300\vskip 0pt plus -50pt }

\def\X{\boldsymbol{X}}

\def\calG{\mathcal{G}}

\def\calS{\mathcal{S}}
\def\calR{\mathcal{R}}

\def\calS{\mathcal{S}}

\def\onevec{\boldsymbol{1}}

\def\hvec{\boldsymbol{h}}

\def\xvec{\boldsymbol{x}}

\def\betavec{\boldsymbol{\beta}}

\def\trans{^{\mathrm{T}}\!}

\def\E{\operatorname{\mathbb{E}}}

\def\Reals{{\mathbb{R}}}

\def\Naturals{{\mathbb{N}}}

\let\le=\leqslant
\let\leq=\leqslant
\let\ge=\geqslant
\let\geq=\geqslant

\numberwithin{equation}{section}

\def\dfrac#1#2{\lower0.15ex\hbox{\large$\textstyle\frac{#1}{#2}$}}
\def\({\bigl(}
\def\){\bigr)}

\def\Naturals{{\mathbb{N}}}
\let\eps\varepsilon

\numberwithin{equation}{section}

\title{Sprinkling with random regular graphs}
\author{
	Mikhail Isaev\thanks{Research supported by ARC DP220103074.}\\
	Monash University\\
	\tt mikhail.isaev@monash.edu \and
	Brendan D. McKay\\
	Australian National University\\
	\tt brendan.mckay@anu.edu.au
	\and
	Angus Southwell\\
	Quantum for NSW\\
	\tt angusjsouthwell@gmail.com
	\and 
	Maksim Zhukovskii\\
	The University of Sheffield\\
	\tt m.zhukovskii@sheffield.ac.uk
}

\date{}

\begin{document}
	\maketitle
	
	\begin{abstract}
	We conjecture that the distribution of the edge-disjoint union of two random regular graphs on the same vertex set is asymptotically equivalent to a random regular graph of the combined degree, provided  it  grows as the number of vertices tends to infinity.
We verify this conjecture for the cases when the graphs are sufficiently dense or sparse. We also prove an asymptotic formula for the expected number of spanning regular subgraphs  in a random regular graph.
	\end{abstract}

 

\section{Introduction}\label{S:intro}

The proof technique known as \emph{sprinkling} has been used since the earliest
days of random graph theory~\cite{AKS1982,ER1960}.
This technique is also known under the name \emph{multiple-round exposure}  
as it 
gradually exposes the edges of a random graph in rounds to achieve a desired graph property through additional last-minute randomisation.  

Sprinkling  has numerous applications in the study of monotone properties of random graphs such as the appearance of giant and dominant components, see~\cite{AKS1982, BCHSS2006} and~\cite[Section~11.9]{AS2016}, 
the existence of a Hamiltonian cycle~\cite[Section~8.2]{Bollobas2001}, the square of a Hamiltonian cycle~\cite{KNP2021}
and  also Ramsey properties, see~\cite{RR1998} and \cite{FRS2010}. Sprinkling is  a crucial proof technique for non-monotone properties related to
extremal subgraphs~\cite[Theorem 1.4]{BPKS2012} and extension counts \cite{Spencer1990}.    
Furthermore, this technique was used to establish the Sunflower Lemma~\cite[Section 2.1]{AlWZ2021}, the Spread Lemma~\cite[Section 2.4]{MNSZ2022}, and also appeared in the proof of recently resolved Kahn--Kalai conjecture \cite{PP2022}.

To our knowledge, the fruitful idea of  sprinkling is undeveloped in the study of the properties
of 
the random regular graph model $\calG(n,d)$. Recall that in this model the graph is taken uniformly from the set of labelled $d$-regular graphs on the vertex set $[n]:=\{1,\ldots,n\}$. If $\min\{d, n-d\}\rightarrow \infty$ we believe that sprinkling is possible in the following sense: 
with probability tending to 1, 
a uniform random regular graph can be split into a union of two edge-disjoint uniform random regular graphs.

To be more formal, given $d_1,d_2  \in \Naturals:=\{0,1,\ldots\}$, let $\calS_{n}(d_1,d_2)$ denote the set of pairs of edge-disjoint graphs $(G_1,G_2)$ on the same vertex set $[n]$ such that $G_1$ is a $d_1$-regular graph and $G_2$ is a $d_2$-regular graph. Throughout the paper, all asymptotic statements are with respect to $n\rightarrow \infty$ and we always implicitly assume that $nd_1$ and $nd_2$  are even to ensure the 
existence of regular graphs.

\begin{conjecture}\label{con1}
Let $d_1=d_1(n)\in \Naturals$ and $d_2=d_2(n)\in \Naturals$ satisfy  
\[	
d_1+d_2 \leq n-1, \qquad
\min\{d_1 + d_2,n-d_1, n-d_2\}  \rightarrow \infty.
\]
Let $(\G_1,\G_2)$ be a uniform random element of $\calS_n (d_1,d_2)$. Then, there is a coupling 
$
(\G_1,\G_2,\G_{d_1},\G_{d_2},\G_{d_1+d_2})
$
such that
\[
\Pr (\G_{d_1 } = \G_1,  \G_{d_2 } = \G_2, \G_{d_1 +d_2} = \G_1\cup \G_2) = 1-o(1)
\]
and  $\G_{d_1} \sim \calG(n,d_1)$, $\G_{d_2} \sim \calG(n,d_2)$, $\G_{d_1+d_2} \sim \calG(n,d_1+d_2)$.
\end{conjecture}

Note  that  $\G_{1}\cup\G_2$  defined in Conjecture \ref{con1} has the same distribution as the   random graph    $\G_{\oplus} \sim \calG(n,d_1)\oplus\calG(n,d_2)$      obtained by	sampling independently random regular graphs $\G_{d_1}\sim \calG(n,d_1)$ and $\G_{d_2}\sim \calG(n,d_2)$  
conditioned on the event that $\G_{d_1}$ and $\G_{d_2}$  are disjoint.
When $d_1$ and $d_2$ are both fixed positive integers and
$(d_1,d_2)\neq (1,1)$, it is known that the random graph model 
$\calG(n,d_1)\oplus\calG(n,d_2)$   is contiguous to $\calG(n,d_1+d_2) $;
see  \cite[Section~4]{Janson1995} and \cite[Section~4.3]{Wormald1999}. Recall that ``contiguity'' means that if events have vanishing probabilities in one model as $n\rightarrow \infty$ then they also do in the other model.
Our second conjecture is that if $d_1+d_2 \rightarrow \infty $ then these models have a much stronger relationship: $\calG(n,d_1)\oplus\calG(n,d_2)$  is asymptotically equivalent to  $\calG(n,d_1+d_2)$.

\begin{conjecture}\label{con11}
If    $d_1+d_2 \leq n-1$ and   $d_1 + d_2 \rightarrow \infty $ then there is a coupling $(\G_{\oplus},\G_{d_1+d_2})$
such that $\G_{\oplus}\sim \calG(n,d_1) \oplus \calG(n,d_2)$, $\G_{d_1+d_2}\sim \calG(n,d_1+d_2)$, and 
\[
\Pr(\G_{\oplus} = \G_{d_1+d_2}) = 1-o(1).
\]
\end{conjecture}

In this paper, we confirm both conjectures   for the cases when degrees $d_1,d_2$ are sufficiently sparse or sufficiently dense.
Everywhere in the paper
all logarithms are   natural logarithms.
\begin{theorem}\label{T:main1}
Conjecture \ref{con1} and Conjecture  \ref{con11}  hold  in the following three cases:
\begin{itemize}\itemsep=0pt
\item[(i)] $d_1 =1$, $d_2 \rightarrow \infty$, and $d_2 = o(n^{1/3})$;
\item[(ii)] $\min\{d_1,d_2, n-d_1 -d_2\} = \omega\Bigl(\Dfrac{n}{\log n}\Bigr)$;
\item[(iii)] $d_1^2\Bigl(\Dfrac{n}{\min\{d_2,n-d_2\}}\Bigr)^{\!2d_1}\leq\dfrac{1}{3}
\log n$.  
\end{itemize}
\end{theorem}
In particular, condition (iii) of Theorem \ref{T:main1}  holds in the case when $d_1=o(\log\log n)$ and $\min\{d_2,n-d_2\}= \Theta(n)$. For smaller $d_1$, we can have sublinear $d_2$ or $n-d_2$.

Theorem~\ref{T:main1}(i)  follows 
by combining the estimates of \cite{Gao2023,McKay1985} and the  Reduction lemma stated below as Lemma \ref{l:reduction}; see Section \ref{S:Thm_i}  for details.
The proofs of parts (ii) and (iii) of Theorem \ref{T:main1}  are less straightforward 
and are given in Section~\ref{S3} and 
Section~\ref{S:dense-sparse}, respectively.

%

\begin{remark}
It might seem  that Conjecture \ref{con11} follows from from Conjecture \ref{con1} by combining  couplings 
$(\G_{1}, \G_{d_1})$, $(\G_{2}, \G_{d_2})$, and $(\G_{1}\cup\G_2,\G_{d_1+d_2})$. This is the case under assumptions of   Theorem \ref{T:main1}. However, in general, this is actually other way around.  
Note that unlike the first conjecture, 
Conjecture \ref{con11}   requires neither  $n-d_1\rightarrow \infty $ nor $n-d_2 \rightarrow \infty$. One can derive 
Conjecture \ref{con1}  from Conjecture \ref{con11} by exchanging the roles of
$(d_1,d_2,n-1-d_1-d_2)$ using the complements of the respective graphs.
\end{remark}

For a graph $G$ and $h\in\Naturals$, let
\[
\calR_{h}(G):= \text{the set of $h$-regular spanning subgraphs of $G$}.
\]
%
\begin{lemma}[Reduction lemma]\label{l:reduction}
Conjecture \ref{con11}  holds if and only if,	for any fixed $\eps>0$,
\[
\Pr\Big(|\calR_{d_1}(\G_{d_1+d_2})| \leq (1-\eps) \E |\calR_{d_1}(\G_{d_1+d_2})|\Big)\rightarrow 0.
\]
\end{lemma}
We prove Lemma \ref{l:reduction} in Section  \ref{S:reduction} using  sufficient conditions on coupling existence (in a general setting) stated in Section \ref{S:coupling}, which is of independent interest.

The asymptotics of $\E |\calR_{d_1}( \G_{d_1+d_2})|$ in the general case is  another open question, which does not imply nor is implied by Conjectures \ref{con1} and \ref{con11}.
For the case when $d_1+d_2 = n-1$,
it is known that, for all $1\leq  h \leq n-2$,
\begin{equation}\label{eq:reg}
|\calR_h(K_n)| = (2^{1/2} e^{1/4}
+o(1)) 
\left(\Dfrac{h}{n-1}\right)^{hn/2}
\left(1-\Dfrac{h}{n-1}\right)^{(n-1-h)n/2}
\binom{n-1}{h }^n,
\end{equation}
where $K_n$ is the complete graph with $n$ vertices; see \cite{LW2017,MW1,MW2}.
Given \eqref{eq:reg}, the computation of  $\E |\calR_{d_1}( \G_{d_1+d_2})|$  
is equivalent to estimating $|\calS_n(d_1,d_2)|$ or 
the probability that two random graphs $\G_{d_1} \sim \calG(n,d_1)$ and 
$\G_{d_2} \sim \calG(n,d_2)$ sampled independently are disjoint.

Hasheminezhad and McKay~\cite{HM1} conjectured an asymptotic formula for the
number $R(n;d_0,\ldots,d_k)$ of ways to partition the complete graph into edge-disjoint
regular subgraphs of degrees $d_0,\ldots,d_k$, when $k=o(n)$.
We present it in Section~\ref{s:related}. 
Since
\[
\E|\mathcal{R}_{d_1}(\G_{d_1+d_2})|=\Dfrac{R(n;n-d_1-d_2-1,d_1,d_2)}{|\calR_{d_1+d_2}(K_n)|} = \Dfrac{|\calS_n(d_1,d_2)|}{|\calR_{d_1+d_2}(K_n)|},
\]
a special case of that conjecture is as follows.

\begin{conjecture}\label{con2} 
If $d_1,d_2>0$ and $d_1 + d_2 \leq n-1$ then
\[
\E |\calR_{d_1}(\G_{d_1+d_2})|
= (2^{1/2} e^{1/4} +o(1)) \bigl(  \lambda ^{\lambda }(1-\lambda)^{1-\lambda }\bigr)^{n (d_1+d_2)/2}
\binom{d_1+d_2}{d_1}^{n},
\]
where $\lambda:= \dfrac{d_1}{d_1 +d_2}$. 
\end{conjecture}

Note that Conjecture \ref{con2} generalises~\eqref{eq:reg}.
This conjecture is confirmed by previous results
for the ranges of degrees considered in Theorem \ref{T:main1}(i,iii); see the discussion in the next section after Conjecture \ref{con_HM}. Our contribution towards Conjecture~\ref{con2} is the following result.

\begin{theorem}\label{T:main2}
Conjecture \ref{con2} holds 
if $d_1,d_2 = \omega\left(\dfrac{n}{\log n}\right)$.
\end{theorem}
We prove Theorem \ref{T:main2} together with Theorem \ref{T:main1}(ii)
in Section \ref{S:main2}.
The proof relies on the aforementioned reduction lemma and the subgraph asymptotic enumeration formula  from \cite[Theorem 10]{GIM2022}. 

\subsection{Related results and conjectures}\label{s:related}
Several results and conjectures should be mentioned in the context of Conjecture  \ref{con1}, Conjecture \ref{con11}, and  Conjecture \ref{con2}. 
In particular,
the following embedding of random regular graphs was  conjectured  in  \cite[Conjecture 1.2]{GIM2022}.

\begin{conjecture}[Gao, Isaev, McKay, \cite{GIM2022}]\label{con:emb}
Let $0 \leq d_1 \leq d \leq  n-1$ be integers, other than 
$(d_1, d) = (1, 2)$ or
$(d_1, d) = (n-3, n-2)$.  Then there exists a coupling
$(\G_{d_1}, \G_d)$ such that $\G_{d_1} \sim \calG(n, d_1)$, $\G_d \sim \calG(n, d)$, and $\Pr(\G_{d_1} \subseteq \G_d) = 1-o(1)$.
\end{conjecture}

Prior to this work, Conjecture \ref{con:emb} was known to be true for the following cases; see \cite[Section 1.2]{Gao2023} and references therein for details.
\begin{itemize}\itemsep=0pt
\item $d_1 =1$ and $3\leq d \leq n-1$;
\item $d_1 = \omega(\log^7 n)$, $d_1 = O\left(\dfrac{n^{1/7}}{\log n}\right)$, and $d_1\leq d \leq n-1$;
\item $d_1\rightarrow \infty$, $d_1 = o(n^{1/3})$, and $d= d_1+1$;
\end{itemize}
A recent sandwiching result \cite[Theorem 1.2]{GIM2023} also implies Conjecture \ref{con:emb} when
\begin{itemize}
\item $n-\log^4 n \geq  d\geq \gamma d_1 \geq \log^4 n$ for some fixed $\gamma>1$.
\end{itemize}
As an immediate corollary of Theorem \ref{T:main1}(ii), we get the  following result, which is not covered by the previous results.

\begin{corollary}
Conjecture \ref{con:emb} is true if $\min\{d_1,d-d_1\} = \omega\left(\dfrac{n}{\log n}\right)$.
\end{corollary}

Conjecture \ref{con:emb} is weaker than Conjecture \ref{con1}. In particular, the former only requires two random regular graphs $\G_{d_1}$ and $\G_{d}$ to be uniform but not their difference.
Note that in the case when $d=d_1+1$, a random relabelling of vertices ensures that the perfect matching $\G_{d_1+1}\setminus\G_{d_1}$ has the uniform distribution.
However, this still does not imply Conjecture \ref{con1} since the random pair $(\G_1,\G_2)$ from $\calS_n$ can have different marginal distributions.

Given  that 
Conjecture \ref{con:emb} is stated (and confirmed for some cases) as including constant degrees, 
one can be tempted to believe that Conjectures \ref{con1} and \ref{con11} might also hold for $d_1+d_2 = O(1)$. However, this is not true. In particular, it is known that the number of perfect matchings is not concentrated in $\G(n,d)$ for any fixed $d$: it follows, for example, from the asymptotic distribution derived in \cite[Theorem~5]{Janson1995}. Applying the Reduction lemma (Lemma \ref{l:reduction}) we can see that the assumption 
$d_1+ d_2 \rightarrow \infty$ is necessary, at least if $d_1 =1$.

Conjecture~\ref{con2} is a particular case of a more general conjecture from~\cite{HM1} about the number of partitions of a clique into spanning regular disjoint graphs stated below.
For integers $d_0,d_1,\ldots,d_k\geq 1$ such that $\sum_{i=0}^k d_i=n-1$, recall that $R(n;d_0,\ldots,d_k)$ is the number of ways to partition the edges of $K_n$ into spanning regular subgraphs of degrees $d_0,d_1,\ldots,d_k$. Let $\lambda_i=\frac{d_i}{n-1}$.

\begin{conjecture}[Hasheminezhad, McKay, \cite{HM1}]\label{con_HM}
If $k=o(n)$, then
\[
R(n;d_0,\ldots,d_k)=(1+o(1))e^{k/4} 2^{k/2}\biggl(\,\prod_{i=0}^k \lambda_i^{\lambda_i}\biggr)^{ \binom{n}{2}}
\binom{n-1}{d_0,\ldots,d_k}^n. 
\]
\end{conjecture}

The case $k=1$ of Conjecture \ref{con_HM} is equivalent to \eqref{eq:reg} and thus resolved completely by combining the results of~\cite{LW2017,MW1,MW2}.  
Conjecture~\ref{con_HM} is also confirmed in \cite{HM1} for the following three cases: 
\begin{itemize}\itemsep=0pt
\item $\sum_{i=1}^k d_i=o(n^{1/3})$ and $\sum_{1\leq i\leq j<t\leq k}d_id_jd_t^2=o(n)$,
\item $k=o(n^{5/6})$ and $d_1=\cdots=d_k=1$,
\item $\min\{d_1,n-1-d_1\}\geq cn/\log n$ for some constant $c>\frac23$ and $\sum_{i=2}^k d_i=O(n^{\eps})$ for sufficiently small $\eps>0$.
\end{itemize}

Theorem \ref{T:main2} establishes this conjecture for the case $k=2$ and
$\min\{d_0, d_1,d_2\}\gg n/\log n$, which is not covered by previous results. We believe that the methods given in Section~\ref{S3} of this paper  can be extended to prove 
Conjecture \ref{con_HM} in the case of an arbitrary fixed number of parts $k$, provided that all the factors are sufficiently dense.



\section{Coupling construction and  the Reduction lemma}
In this section, we consider a general question of existence of a coupling with given marginal distributions satisfying constraints encoded via a bipartite graph with high probability.
 As a consequence, we establish Lemma \ref{l:reduction} (the Reduction lemma). Finally, we derive Theorem \ref{T:main1}(i) based on estimates from \cite{Gao2023} and \cite{GIM2021}.

\subsection{General coupling for a bipartite graph}\label{S:coupling}

Let $D$ be a bipartite graph with parts $S$ and $T$. For simplicity we identify $D$ with its set of edges from $S \times T$.
Koperberg \cite[Proposition 6]{Koperberg2024}
gives a proof of the following result.

\begin{theorem}[Strassen’s Theorem with deficiency \cite{Koperberg2024}]\label{T:couplingNew}
Let $\epsilon \in [0,1]$.
There exists a coupling $(X,Y)\in S\times T$ with marginal distributions $\pi_S$ and $\pi_T$ such  
that
$\Pr(XY \notin D)\leq \epsilon$   if and only if 
\begin{equation}\label{eq:Hall}
    \pi_S(\varOmega) \leq \pi_T(N(\varOmega)) + \epsilon \quad \text{for any $\varOmega \subseteq S$,}
\end{equation}
where $N(\varOmega)\subseteq T$ is the set of vertices that are adjacent in $D$ to a vertex from $\varOmega$. 
\end{theorem}
Condition \eqref{eq:Hall} resembles Hall's classical marriage theorem. For the case  of $\epsilon=0$ it appeared in Strassen’s theorem \cite{Strassen1965} which applies to a more general non-discrete setting with Polish spaces. Thus, Theorem \ref{T:couplingNew} is regarded as Strassen’s theorem with deficiency~$\eps$. Similar results appeared in many previous works; see \cite{Koperberg2024} for more references and discussions. In particular,
Theorem \ref{T:couplingNew} can   be derived as a consequence of Mirsky’s theorem  on integral matrices with constraints \cite{Mirsky1968}.

As a consequence of Theorem \ref{T:couplingNew}, we give the following sufficient conditions which are easier to check for our purposes.


\begin{corollary}\label{T:coupling}
Let $\delta,\eps \in [0,1]$ and
\begin{align*}
\GoodS &:= \left\{x\in S \st  \deg_D(x) \geq (1-\eps)\dfrac{|D|}{|S|}\right\},\\
\GoodT &:= \left\{y\in T \st  \deg_D(y) \geq (1-\eps)\dfrac{|D|}{|T|}\right\}.
\end{align*}
Assume that $|\GoodS| \geq (1-\delta)|S|$ and $|\GoodT| \geq (1-\delta)|T|$.
Then, there is a coupling $(X,Y)$ such that
$X, Y$ are uniformly distributed on $S$ and $T$, respectively, and
\[
\Pr(XY \notin D) \leq   2\delta + \Dfrac{\eps}{1-\eps}.  
\]

\end{corollary} 

\begin{proof}
Let $\pi_S$ and $\pi_T$ denote the uniform measures on $S$ and $T$, respectively.
For any $\varOmega \subseteq S$,  by the assumptions and the definition of $\GoodS$, we have
\[
   \frac{1}{|D|} \sum_{x\in \varOmega} \deg_D(x)
    \geq (1-\eps)\frac{|\varOmega \cap \GoodS|}{|S|}  
    \geq (1-\eps) (\pi_S(\varOmega) -\delta).
\]
Similarly, 
swapping the roles of $S$ and $T$ and using $\varOmega' = T\setminus N(\varOmega)$,
we derive that
\[
\frac{1}{|D|} 
\sum_{y\in T\setminus N(\varOmega)} \deg_D(y)
\geq (1-\eps) ( \pi_T(\varOmega') -\delta)
=(1-\eps) ( 1- \pi_T(N(\varOmega)) -\delta).
\]
Since an edge of $D$ incident to a vertex of $\varOmega$  cannot be incident to a vertex from $T\setminus N(\varOmega)$, we derive that
\[
    \sum_{x\in \varOmega} \deg_D(x) +
    \!\sum_{y\in T\setminus N(\varOmega)}\! \deg_D(y) \leq |D|.
\]
Combining the three displayed equations above, we find that 
\[
1- 2\delta + \pi_S(\varOmega) - \pi_T(N(\varOmega)) \leq \Dfrac{1}{1-\eps}.
\]
This is equivalent to condition \eqref{eq:Hall} with 
$\epsilon := 2\delta+ \dfrac{\eps}{1-\eps}$. Applying Theorem \ref{T:couplingNew} completes the proof.
 \end{proof}

We use Corollary \ref{T:coupling}  to prove Lemma \ref{l:reduction} (the Reduction lemma) in the next section.  
 However, this result is of independent interest. For example, its extended version  (with slightly less optimal constants) was used to establish the distribution of the maximal number of common neighbours in random regular graphs; see \cite[Theorem 5.1]{IZ2024}.

\subsection{Proof of Lemma~\ref{l:reduction}}\label{S:reduction}


Let $d:=d_1+d_2$. Consider a bipartite graph $D$ with parts $S=\mathcal{R}_d(K_n)$, which is the set of all $d$-regular graphs on $[n]$, and $T=\mathcal{S}_n(d_1,d_2)$, and the set of edges defined as follows: $G\in S$ is adjacent to all nodes $(G_1,G_2)\in T$ such that $G_1\cup G_2=G$. 

Assume first that the concentration condition from the Reduction lemma holds. This is equivalent to the existence  of $\bar\eps=\bar\eps(n)=o(1)$ such that $|\calR_{d_1}(\G_{d})|  \geq (1-\bar\eps) \E |\calR_{d_1}(\G_{d})|$ with probability tending to 1. Let $\delta=\delta(n)=o(1)$ be the fraction of $d$-regular graphs $G$ on $[n]$ such that $|\calR_{d_1}(G)|<(1-\bar\eps) \E |\calR_{d_1}(\G_{d})|$.
Clearly, every node from $T$ has degree 1 implying that $d_T=1$ and $\GoodT=T$.
Moreover, every node $G\in S$ has degree $|\mathcal{R}_{d_1}(G)|$ and thus $d_S=\E |\calR_{d_1}(\G_{d})|$.
We then get $|\GoodS|=(1-\delta)|S|$.
Corollary~\ref{T:coupling} gives the desired coupling of a uniformly random $X:=\G_{d}\sim\mathcal{G}(n,d)$ in $S$ and a uniformly random $Y=(Y_1,Y_2)\stackrel{d}=(\G_1,\G_2)$ in $T$ such that 
$$
\Pr(Y_1\cup Y_2\neq X)=\Pr(XY\notin D)\leq 2\delta + \dfrac{\bar\eps}{1-\bar\eps}=o(1).
$$

For the opposite direction, assume that $(X=\G_{d_1+d_2},Y=(\G_1,\G_2))$ satisfies 
$$
\Pr(X\neq\G_1\cup\G_2)=\Pr(XY\notin D)=o(1).
$$
Take any constant $\eps>0$ and consider the respective set 
$\GoodS=\{G\in \mathcal{G}(d_1+d_2)\st |\mathcal{R}_{d_1}(G)|\geq(1-\eps)d_S\}$. We have that
\begin{align*}
\Pr(X\notin\GoodS) & \leq\Pr(\G_1\cup\G_2\notin\GoodS)+\Pr(X\neq\G_1\cup\G_2)\\[1ex]
& =\Dfrac{|\{(G_1,G_2)\in\mathcal{S}_n(d_1,d_2)\st |\mathcal{R}_{d_1}(G_1\cup G_2)|<(1-\eps)d_S\}|}{|\mathcal{S}_n(d_1,d_2)|}+o(1)\\
&=\sum_{G\notin \GoodS}\Dfrac{|\mathcal{R}_{d_1}(G)|}{|D|}+o(1)<(1-\eps)\sum_{G\notin \GoodS}\Dfrac{d_S}{|D|}+o(1)\\
&=(1-\eps)\frac{|S|-|\GoodS|}{|S|}+o(1)=
(1-\eps)\Pr(X\notin\GoodS)+o(1).
\end{align*}
Therefore, $\Pr(X\notin\GoodS)=o(1)$ as required.

\subsection{Proof of Theorem \ref{T:main1}(i)}\label{S:Thm_i}


Recall our  assumptions that  $d_1=1$, $d_2=o(n^{1/3})$, and $d_2\rightarrow \infty$. Let $d_3:=n-1-d_1-d_2$.
Using   Lemma \ref{l:reduction},  we first observe that it is sufficient to establish that, for any fixed $\eps>0$,  
\begin{align}
\label{eq:first}     |\calR_{d_1}(\G_{d_1+d_2})| &\geq (1-\eps) \E |\calR_{d_1}(\G_{d_1+d_2})|,\\
\label{eq:second}      |\calR_{d_2}(\G_{d_2+d_3})| &\geq (1-\eps) \E |\calR_{d_2}(\G_{d_2+d_3})|, \\
|\calR_{d_1}(\G_{d_1+d_3})| &\geq (1-\eps) \E |\calR_{d_1}(\G_{d_1+d_3})|
\label{eq:third}  
\end{align}
are events with probability tending to $1$. Here,  $\G_{d_2+d_3}\sim \calG(n,d_2+d_3)$
and $\G_{d_1+d_3}\sim \calG(n,d_1+d_3)$.
Indeed, if these hold then Lemma \ref{l:reduction} implies the existence of couplings 
$(\G_1 \cup \G_2, \G_{d_1+d_2})$, $(K_n \setminus \G_1, \G_{d_2+d_3})$, and 
$(K_n \setminus \G_2, \G_{d_1+d_3})$ such that
\begin{align*}
\Pr(\G_1 \cup \G_2 = \G_{d_1+d_2}) = 1-o(1),\\
\Pr(K_n \setminus \G_1 = \G_{d_2+d_3}) = 1-o(1),\\
\Pr(K_n \setminus \G_2 = \G_{d_1+d_3}) = 1-o(1).
\end{align*} 
In particular,  
 we get $\G_1 = K_n \setminus \G_{d_2+d_3} \sim \calG(n,d_1)$
 and $\G_2 =  K_n \setminus \G_{d_1+d_3} \sim \calG(n,d_2)$ with probability  $1-o(1)$.
Thus, combining these couplings we  get  the required coupling $(\G_1,\G_2,\G_{d_1},\G_{d_2},\G_{d_1+d_2})$.

Note that  \eqref{eq:second} is immediate since $d_2+d_3 = n-2$ and all $(n-2)$-regular graphs are isomorphic. 
Recalling that $d_2\rightarrow \infty$, \eqref{eq:first} 
follows from Chebyshev's inequality and the estimates of \cite{Gao2023} stated below for convenience.
\begin{theorem}[Gao~\cite{Gao2023}]\label{G2023}
If $d\geq 3$ and $d=o(n^{1/2})$ then
$$
\E |\mathcal{R}_{1}(\G_d)|^2=\biggl(1+\Dfrac{1}{6d^3}+O\Bigl(d^{-4}+\Dfrac{d^3}{n}
+\sqrt{\dfrac{d}{n}}\log^3n\Bigr)\biggr)\bigl(\E |\mathcal{R}_{1}(\G_d)|\bigr)^2.
$$
\end{theorem}
%

Finally, observe that $|\mathcal{R}_{d_1}( \G_{d_1+d_3})|$ is the number of perfect matchings that avoid edges of $\G_{d_2}$.  Then,  \eqref{eq:third} is straightforward  from the result stated below, which  is the special case of  \cite[Theorem 4.6]{McKay1985} for the regular case and small degrees.

\begin{theorem}[McKay~\cite{McKay1985}]\label{M1985}
Let $G_1$ be a $d_1$-regular graph and $G_2$ be a $d_2$-regular graph on the same vertex set $[n]$. Let $d_1(d_1+d_2)=o(\sqrt{d_1 n})$. Then the number of graphs on $[n]$ isomorphic to $G_1$ that avoid edges of $G_2$ equals
$$
\frac{(nd_1)!}{(nd_1/2)! \, 2^{nd_1/2}(d_1!)^n}
\exp\biggl(-\Dfrac{d_1-1}{2}-\Bigl(\Dfrac{d_1-1}{2}\Bigr)^2-\Dfrac{d_1d_2}{2}+o(1)\biggr).
$$
\end{theorem}
%




\section{Sprinkling dense with dense}\label{S3}

In this section,  we consider the dense case when all degrees (and the complement degrees) are $\omega\Bigl(\dfrac{n}{\log n}\Bigr)$ and give the proofs of  Theorem \ref{T:main1}(ii) and  Theorem \ref{T:main2}.
We use much of the machinery of \cite{GIM2022}, specifically the results on asymptotic  enumeration of subgraphs with a given degree sequence $\hvec$ of a sufficiently dense graph $G$. For our purposes, we only need to consider $d$-regular $G$ and regular $\hvec= (h,\ldots,h)^T$. 

\subsection{Preliminaries}
We start by recalling the definitions from \cite{GIM2022}. Let 
\[
\drat :=  \frac{1}{dn}  \sum_{j\in [n]} h_j,
\]
which is the relative density of a subgraph with degree sequence $\bm{h}=
(h_1,\ldots,h_n)^T$ in a graph $G$. 
The following system of equations for $\betavec\in \Reals^n$ is of crucial importance for the asymptotic enumeration of 
 such subgraphs in \cite[Section 5]{GIM2022}:
\begin{align}\label{eqn:system}
h_j = \sum_{k \st jk \in G} \Dfrac{e^{\beta_j + \beta_k}}{1 + e^{\beta_j + \beta_k}} \quad \mbox{for all $j \in [n]$.}
\end{align}
This system expresses that a random subgraph with independent edge probabilities  $\dfrac{e^{\beta_j + \beta_k}}{1 + e^{\beta_j + \beta_k}}$ 
 has expected degree sequence $\hvec$.
For the regular case when $h_j = h$ for all $j \in [n]$,  to solve \eqref{eqn:system}, we can take $\bm{\beta}=(\beta,\ldots,\beta)$ 
with $\beta := \frac{1}{2}\log \frac{\drat}{1-\drat}$.
Define
\begin{equation}\label{def:uv}
\begin{aligned}
u(\bm{\theta}) &= \dfrac{1}{24} \drat(1-\drat)(1 - 6\drat + 6\drat^2) \sum_{jk \in G} (\theta_j + \theta_k)^4, 
\\[-1ex]
v(\bm{\theta}) &= \dfrac{1}{6} \drat (1-\drat)(1-2\drat)\sum_{jk \in G}(\theta_j +\theta_k)^3. 
\end{aligned}
\end{equation}
Recall that the signless Laplacian matrix $Q=Q(G)$ of a graph $G$ is defined by 
\begin{align*}
\bm{\theta}^T \!{Q} \bm{\theta} =  \sum_{jk \in G} (\theta_j + \theta_k)^2.
\end{align*}

The following result is a special case of \cite[Theorem 11]{GIM2022} for the regular case. 
Note that the theorem in~\cite{GIM2022} is stated for the Gaussian vector
$\X' = \Bigl( \dfrac{2}{\lambda(1-\lambda)}\Bigr)^{1/2}\X$,
which has density proportional to
$e^{-\frac12 \lambda(1-\lambda)\xvec^T\! Q \xvec}$,
instead of the vector $\X$ used in this paper, so its moments have to be scaled correspondingly. Also, since we only consider the regular case, there is no need to assume the existence of a solution of the system \eqref{eqn:system}:  we just take $\betavec =  (\beta,\ldots, \beta) \in \Reals^n$ described above.

\begin{theorem}[Gao, Isaev, McKay, \cite{GIM2022}] \label{thm:jane-misha-brendan}
Let $\eps$ and $\gamma$ be fixed positive constants. Let  $nh$ and $nd $ be even for some positive integers $1\leq h \leq d<n$.  Suppose a $d$-regular graph $G$ satisfies the following assumptions:
\begin{enumerate}\itemsep=0pt
\item[(i)] for any two distinct vertices $j$ and $k$, we have 
\[
\Dfrac{\gamma d^2}{n} \leq |\{\ell \st j \ell \in G \mbox{ and } k \ell \in G\}| \leq \Dfrac{d^2}{\gamma n};
\]
\item[(ii)] $\drat(1 - \drat) d  = \omega\left(\dfrac{n}{\log n}\right)$, where $\drat = h/d$.
\end{enumerate}
Let $\bm{X}$ be a random variable with  density $\pi^{-n/2} (\det{Q})^{1/2} e^{-\xvec^T {Q} \xvec}$.
Then, as $n \rightarrow \infty$, 
\begin{align*}
|\calR_h(G)| = \Dfrac{2   \left(\lambda^{\lambda }  (1 - \drat)^{1-\lambda}\right)^{-dn/2} }{(2\pi \drat(1-\drat))^{n/2}    (\det Q)^{1/2}}
\exp \biggl( \Dfrac{4\,\Exp{u(\bm{X})}}{\drat^2(1-\drat)^2} - \Dfrac{4\,\Exp{v^2(\bm{X})}}{ \drat^3(1-\drat)^3} + O(n^{-1/2 + \eps}) \biggr),
\end{align*}
where the constant implicit in $O(\cdot)$ depends on $\gamma$ and $\eps$ only.
\end{theorem}
Our proof of Theorem \ref{T:main1}(ii) and  Theorem \ref{T:main2} relies on  the following corollary of Theorem \ref{thm:jane-misha-brendan}, which we establish in the next two subsections.
\begin{corollary}\label{cor-jmb}
Let   $nh$ and $nd $ be even and $\drat(1 - \drat) d  = \omega\Bigl(\dfrac{n}{\log n}\Bigr)$. Then
\begin{itemize}\itemsep=0pt
\item[(a)] If $d\geq \alpha n$ for some $\alpha>\frac12$ then 
\[
|\calR_h(G)| = \Dfrac{  \left(\lambda^{\lambda }  (1 - \drat)^{1-\lambda}\right)^{-dn/2} }{(2\pi d \drat(1-\drat))^{n/2}  }  e^{o(\log n)}.
\]
\item[(b)] If all pairs of distinct vertices of $G$ have 
$\dfrac{d^2}{n} + o\left(\dfrac{d^3}{n^2}\right)$ common neighbors then 
\[
|\calR_h(G)| = \Dfrac{   \left(\lambda^{\lambda }  (1 - \drat)^{1-\lambda}\right)^{-dn/2 }}
{(2\pi d \drat(1-\drat))^{n/2}} 2^{\frac12}
\exp\biggl(  \Dfrac{n}{12d} \Bigl(1- \Dfrac{1}{\lambda(1-\lambda)}\Bigr)+\dfrac{1}{4}   + o(1)\biggr).
\]
\end{itemize}
\end{corollary}

Clearly, both assumptions of Corollary \ref{cor-jmb}(a,b) imply assumption (i) of Theorem \ref{thm:jane-misha-brendan}.  
Then, Corollary \ref{cor-jmb}
immediately follows by combining 
Theorem \ref{thm:jane-misha-brendan} and the 
estimates for the determinant of $Q$ given in Lemma \ref{lem:detq} 
and the 
estimates for  the moments $\E[u(\X)]$, $\E[v^2(\X)]$ given in  Lemma \ref{lem:exp-u-v}
and Lemma \ref{lem:exp-u-v2}. To state our estimates in a more explicit form, we introduce the notation 
$ \pm \eps$, which means that there is a number from 
$[-\eps,+\eps]$ such that the considered equality holds.

\subsection{Determinant estimates}

Let $A=A(G)$ denote the adjacency matrix of a $d$-regular graph $G$.
If $Q$ is the signless Laplacian matrix of $G$ then $Q= dI+A$. Therefore,
\begin{equation}\label{Q-lambda}
\det Q = d^n\prod_{i=1}^n (1+\chi_i), 
\end{equation}
where 
$
\chi_1 \geq \chi_2 \geq \cdots \geq  \chi_n 
$
are the eigenvalues of matrix $A/d$. 
The existence of a basis of $n$ orthogonal eigenvectors is ensured by the fact that $A$ is a real symmetric matrix. 
Note also that 
$\onevec = (1,\ldots,1)^T \in \Reals^n$
is the eigenvector corresponding to $\chi_1 =1$. It is also well-known that $\chi_n\geq -1$
with equality if and only if $G$ is bipartite. 
Some spectral properties of   $A/d$ were derived in recent work \cite{Wei2023}, but we could not   estimate $\det Q$ to sufficient precision based on these. Instead, we analyse it more carefully from first principles in the next lemma.

\begin{lemma}\label{lem:detq}
Let $G$ be a $d$-regular graph on $n$ vertices.
\begin{enumerate}
\item[(a)] If $d \geq \alpha n$ for some $\alpha > \frac{1}{2}$, then $\det Q = 2d^n \exp\left( -\frac{1}{2} -\frac{n}{2d} \pm c_\alpha \right)$, where 
\begin{align*}
	c_\alpha := \frac{\left( \frac{1-\alpha}{\alpha} \right)^{3/2}}{1 - \left( \frac{1-\alpha}{\alpha} \right)^{1/2}}.
\end{align*}
\item[(b)] If $n \leq \eps d^2$  for some $\eps \leq  1/4$ and all pairs of vertices in $G$ have $\frac{d^2}{n}\left(1 \pm \eps \right)$ common neighbours, then 
$\det Q = 2d^n\exp\bigl( -\frac{1}{2} -\frac{n}{2d}\pm 4 \eps \bigr)$. 
\end{enumerate}
\end{lemma}

\begin{proof}
It is a well-known fact that 
$\tr(A^r) = \sum_{i=1}^n d^r\chi_i^r$ counts the number of closed walks of length $r$ in the corresponding graph $G$. Recalling that $\chi_1=1$, we get that
\begin{equation}\label{sumlambda2}
\sum_{i=2}^n \chi_i = -1 \quad\mbox{and}\quad \sum_{i=2}^n \chi_i^2 = \Dfrac{nd}{d^2} - 1 = \Dfrac{n}{d}-1.
\end{equation}
Both assumptions (a) and (b) imply that $G$ is not bipartite so 
$ \chi_i \in (-1,1]$ for all $i$.		 
Using \eqref{Q-lambda}, \eqref{sumlambda2}, and replacing $\log (1+\chi_i)$ with its
Taylor series, we get that
\begin{align}
\det Q  
&=
2d^n \exp\biggl(\,  \sum_{i=2}^n\log(1 + \chi_i) \biggr) = 2 d^n \exp\biggl(  - \sum_{i=2}^n\sum_{r \geq 1} \Dfrac{(-1)^r}{r} \chi_i^r \biggr) \notag
\\
&=  2 d^n  \exp\biggl(-\dfrac12  -\Dfrac{n}{2d} - \sum_{r \geq 3} \Dfrac{(-1)^r}{r}\sum_{i=2}^n \chi_i^r \biggr).  \label{Q-lambda-2}
\end{align}
We will estimate the tail of the series separately for (a) and (b).

For (a), from \eqref{sumlambda2}, we get that 
\[
\sum_{i=2}^n \chi_i^2  = \Dfrac{n}{d}-1 \leq \Dfrac{1-\alpha} {\alpha} <1.
\]
Let $\chi': = \max\{|\chi_2|,  |\chi_n| \}$. 
Since $\chi_i \in (-1,1]$, it also follows that
\begin{align*}\label{eqn:eigbound}
\left| \sum_{i=2}^n \chi_i^r \right| 
\leq (\chi')^{r-2} \sum_{i =2}^n \chi_i^2 
\leq \biggl( \sum_{i=2}^n \chi_i^2 \biggr)^{\frac{1}{2}r-1} 
\sum_{i =2}^n \chi_i^2 \leq \left( \Dfrac{1-\alpha}{\alpha}\right)^{\frac{1}{2}r}.
\end{align*}
Thus, we can estimate the tail of the series in \eqref{Q-lambda-2} using a geometric series:
\begin{align*}
\left|\sum_{r \geq 3} \dfrac{(-1)^r}{r}\sum_{i=2}^n \chi_i^r\right|
\leq \sum_{r \geq 3} \left(  \dfrac{1-\alpha}{\alpha}\right)^{\frac12 r}= c_{\alpha}.
\end{align*}
Claim (a) follows.

For part (b), counting the total number of closed  walks of length 3, we get that
\[
d^3\sum_{i =1}^n \chi_i^3 =
\sum_{ ij \in G} \Dfrac{d^2}{n} (1\pm \eps)  = d^3 (1\pm \eps).
\]
Recalling that $\chi_1 =1$, we get that 
$\left|\sum_{i =2}^n \chi_i^3\right|\leq \eps$. 
Similarly, for the number of closed walks of length $4$, we find that
\[
d^4 \sum_{i =1}^n \chi_i^4 
= nd^2  + \sum_{ ijk \in G} \Dfrac{d^2}{n}   (1 \pm \eps)  =
 nd^2 + d^3(d-1) (1\pm \eps)
=
d^4(1\pm 2\eps), 
\]
where the sum is over all ordered paths $ijk \in G$. 
To see this, note that (i) the number of closed walks that visit the starting vertex three times (that is, of the type $ijiki$) is $nd^2$, (ii) the number of ordered paths $ijk$ is $nd(d-1)$, (iii) by assumption all pairs of vertices $i$ and $k$ have $\frac{d^2}{n} \left( 1\pm \eps \right)$ common neighbours, and (iv) by assumption $n\leq \eps d^2$.
 Therefore, 
\[
\sum_{i=2}^n \chi_i^4 \leq 2\eps \leq 1/2.
\]
In particular, we get that
$\chi' = \max\{|\chi_2|,  |\chi_n| \} \leq 2^{-1/4}$. Combining the above estimates gives that
\begin{align*}
\left|\sum_{r\geq 3}\Dfrac{(-1)^r}{r}\sum_{i=2}^n \chi_i^r \right|
\leq  \dfrac{\eps}{3} +
\sum_{r\geq 4}\dfrac{1}{r}\sum_{i=2}^n |\chi_i|^r 
\leq \dfrac{\eps}{3} +
\dfrac14 \sum_{i=2}^n \chi_i^4  \sum_{t \geq 0} 2^{-t/4} \leq 4\eps.
\end{align*}
Combining this with \eqref{Q-lambda-2} completes the proof of part (b).	
\end{proof}

\subsection{Moment estimates}
Recall that 
$\X$ is a random Gaussian vector with density $\pi^{-n/2} (\det Q)^{1/2} e^{-\xvec^T {Q} \xvec}$ and
that $u(\bm{\theta})$ and $v(\bm{\theta})$ are as defined in~\eqref{def:uv}.
Under the assumptions of Corollary \ref{cor-jmb}(a),
the moment estimates of \cite[Lemma 9]{GIM2022} hold as shown in the next lemma.
\begin{lemma}\label{lem:exp-u-v}
If $d \geq \alpha n$ for some $\alpha > \frac{1}{2}$
and $\drat(1 - \drat) d \gg \frac{n}{\log n}$ then 
\[ 
\Dfrac{\Exp{u(\bm{X})}}{\lambda^2(1-\lambda^2)} = o\left( \log n \right)\qquad  \text{and} \qquad
\Dfrac{\Exp{v^2(\bm{X})}}{\lambda^3(1-\lambda^3)} =o \left(\log n\right).
\]
\end{lemma}

\begin{proof}
If $d \geq \alpha n$ for some $\alpha > \frac{1}{2}$ then  $d = \Theta(n)$ and the number of common neighbours of any two vertices is $\Theta(n)$.
Under this condition, the lemma is identical to \cite[Lemma~9]{GIM2022} except that~\cite{GIM2022} uses the scaled Gaussian vector 
$\X' = \Bigl( \dfrac{2}{\lambda(1-\lambda)}\Bigr)^{1/2}  \X$, which has density proportional to
$e^{-\frac12 \lambda(1-\lambda)\xvec\trans Q \xvec}$, instead of the vector~$\X$.
Note that condition (B2) in~\cite{GIM2022} is satisfied by the vector with all entries equal to
$\frac12\log\dfrac{\lambda}{1-\lambda}$, as described above.

Taking this into account, the cited lemma gives
\begin{align*}
\Dfrac{\Exp{u(\bm{X})}}{\lambda^2(1-\lambda^2)} &=    \dfrac14\,\Exp{u(\X')} = O\left(\Dfrac{n}{d\lambda(1-\lambda)}\right) =  o(\log n),\\
\Dfrac{\Exp{v^2(\bm{X})}}{\lambda^3(1-\lambda^3)} &= \dfrac18\,\Exp{v^2(\X')} =  O\left(\Dfrac{n}{d \lambda(1-\lambda)}\right) =  o(\log n),
\end{align*}
as required.
\end{proof}

For Corollary \ref{cor-jmb}(b), we need more precise estimates for 
the moments $\E [u(\X)]$ and $\E [v^2 (\X)]$, which can be expressed via covariances of the components of $\X$ thanks to Isserlis' formula~\cite{Isserlis} stated for convenience below.

\begin{theorem}[Isserlis~\cite{Isserlis}]\label{T:isserlis}
 Let  $\X=(X_1,\ldots,X_n)$ be a random variable with the
  normal density $(2\pi)^{-n/2}\abs{\varSigma}^{-1/2} e^{-\frac12 \xvec\trans\varSigma^{-1}\xvec}$, where  $\varSigma=(\sigma_{jk})$  is a symmetric positive definite matrix.
  Consider a product $Z=X_{j_1}X_{j_2}\cdots X_{j_k}$, where the 
  subscripts do not need to be distinct.  If $k$ is odd, then
  $\E Z=0$.  If $k$ is even, then
  \[
     \E Z = \sum_{(i_1,i_2),(i_2,i_3),\ldots,(i_{k{-}1},i_k)}
        \sigma_{j_{i_1}j_{i_2}}\cdots\sigma_{j_{i_{k{-}1}}j_{i_k}},
  \]
  where the sum is over all unordered partitions of $\{1,2,\ldots,k\}$ into
  $k/2$ disjoint unordered pairs. The number of terms in the sum is
  $(k-1)(k-3)\cdots3\cdot 1$.
\end{theorem}

To apply Isserlis' formula, we first need to estimate the coefficients of $Q^{-1}$.

\begin{lemma}\label{lem:qinv}
Let $G$ be a $d$-regular graph on $n$ vertices 
such that $200\leq n\leq \eps d^2$ for some $\eps \leq 1/4$ and 
every pair of vertices has $\frac{d^2}{n} (1 \pm \eps)$ common neighbours.   If $Q^{-1} = (q_{jk})$, then
\begin{align*}
q_{jk} = 
\begin{cases}
	\frac{1}{d}\left( 1 + \frac{1}{d} - \frac{1}{n} \right)\left(1 + \frac{1}{2n} \pm  \frac{4\eps}{n} \right), &\text{if } j=k; \\
	\frac{1}{d}\left( 1 + \frac{1}{d} - \frac{1}{n} \right)\left( \frac{1}{2n} - \frac{1}{d} \pm \frac{4\eps}{n} \right),  &\text{if } jk \in G; \\
	\frac{1}{d}\left( 1 + \frac{1}{d} - \frac{1}{n} \right)\left(\frac{1}{2n} \pm \frac{4\eps}{n} \right), & \text{otherwise.}
\end{cases}
\end{align*}
\end{lemma} 
\begin{proof}
Recall that $Q = d  I + A$. Since $\onevec$ is an eigenvector of $A$ with eigenvalue $d$, we have 
$QJ = 2dJ$, where $J$ is the $n\times n$ matrix with all entries equal 1. We also have
\begin{equation}\label{Q-start-inverse}
Q \left( I - \Dfrac{A}{d}  + \Dfrac{J}{2n} \right) 
= d  \left(I - \Dfrac{A^2}{d^2} + \Dfrac{J}{n}\right) = \Dfrac{d}{  1 + \frac{1}{d} - \frac{1}{n}  } (I - Y),
\end{equation}
where $Y$ is a matrix of the form
\begin{align*}
Y=
\begin{pmatrix}
	\left(\frac{1}{d} - \frac{1}{n}\right)^2 & \ & \pm\frac{\eps}{n}\left(1 + \frac{1}{d} -\frac1n\right) \\
	\ & \ddots & \ \\
	\pm\frac{\eps}{n}\left(1 + \frac{1}{d}-\frac1n\right) & \ & \left(\frac{1}{d} - \frac{1}{n}\right)^2\\
\end{pmatrix}.
\end{align*}
We show by induction that, for all $r \geq 1$,
\begin{equation}\label{ind-Y}
\|Y^r\|_{\max} \leq \Dfrac{\eps^r}{n}  \left( 1 + \Dfrac{1}{d} \right)^r,
\end{equation} where 
$\|\cdot\|_{\max}$ denotes the maximal absolute value of a matrix entry.

Consider the base case when $r=1$.
The off-diagonal entries are clearly bounded in magnitude by $\dfrac{\eps}{n}\left(1 + \dfrac{1}{d}\right)$. 
Using the assumption that $n \leq \eps d^2$, we get 
\[
\left(\Dfrac{1}{d} - \Dfrac{1}{n}\right)^2 \leq \Dfrac{1}{d^2} \leq \Dfrac{\eps}{n}.
\]
Thus, the required bound also holds for the diagonal entries of $Y$, which proves the base case. We verify the induction step by combining the inequalities $\|Y\|_\infty \leq n \|Y\|_{\max}$ and 
$\|Y^{r+1}\|_{\max} \leq \|Y\|_\infty\, \|Y^{r}\|_{\max}$,
                where $\|Y\|_\infty$ is the matrix norm induced from the vector $\infty$-norm. This proves \eqref{ind-Y}.

Next, we estimate the tail $Z$ of the power series for $(I-Y)^{-1}$
\begin{align*}
Z:= \sum_{r\geq 1} Y^r.
\end{align*}
Note that the assumptions $200\leq n\leq \eps d^2$ and $\eps \leq \frac14$ imply that $1+\frac1d \leq \frac{7}{5}(1-\eps(1+\frac1d))$.
Using \eqref{ind-Y}, it follows that 
\begin{align}\label{eqn:zmax}
\|Z\|_{\max} \leq \sum_{r\geq1}\Dfrac{\eps^r}{n}\left(1 + \Dfrac{1}{d}\right)^r = \Dfrac{1}{n}\cdot 
\dfrac{\eps\left(1 + \frac{1}{d}\right)}{ 1 - \eps\left(1 + \frac{1}{d}\right)} \leq \Dfrac{7\eps}{5n}.
\end{align}
From \eqref{Q-start-inverse}, we get that 
\begin{align}
Q^{-1} &= \Dfrac1d\left( 1 + \Dfrac{1}{d} - \Dfrac{1}{n} \right)\left(I - \Dfrac{A}{d} + \Dfrac{J}{2n}\right)(I+Z) \notag \\&=\label{eqn:matrixinv}
\Dfrac1d\left( 1 + \Dfrac{1}{d} - \Dfrac{1}{n} \right)\left(I - \Dfrac{A}{d} + \Dfrac{J}{2n} +Z - \left(\Dfrac{A}{d} - \Dfrac{J}{2n}\right)Z\right) .
\end{align}
The bounds $\|A\|_\infty \leq d$, $\|J\|_\infty \leq n$, and \eqref{eqn:zmax} imply that
\[
\left\| \left(\Dfrac{A}{d} - \Dfrac{J}{2n}\right)Z\right\|_{\max}
\leq   \left\|\Dfrac{A}{d} - \Dfrac{J}{2n}\right\|_{\infty} \|Z\|_{\max} \leq \dfrac{3}{2}\|Z\|_{\max} \leq \Dfrac{21\eps}{10 n}.
\]
Therefore, 
\[
\left\|Z - \left(\Dfrac{A}{d} - \Dfrac{J}{2n}\right)Z\right\|_{\max}
  \leq \left\|Z\right\|_{\max} + \left\| \left(\Dfrac{A}{d} - \Dfrac{J}{2n}\right)Z\right\|_{\max}
 \leq\Dfrac{4\eps}{n}.
\]
Observe also that $I - \dfrac{A}{d} + \dfrac{J}{2n}$ has entries $1 + \frac{1}{2n}$ on the diagonal, while the off-diagonal entries $ij$ are 
$\frac{1}{2n} - \frac{1}{d}$ for $ij \in G$, and $\frac{1}{2n}$ for other $ij$. 
The lemma now follows from~\eqref{eqn:matrixinv}.
%
%
%
\end{proof}
Now we combine  Theorem \ref{T:isserlis} 
and Lemma \ref{lem:qinv} to estimate $\Exp{u(\bm{X})}$ and $\Exp{v^2(\bm{X})}$
for the case when $d$ is too small for Lemma \ref{lem:exp-u-v}.
\begin{lemma}\label{lem:exp-u-v2}
Let $G$ be a $d$-regular graph on $n$ vertices 
such that $n\leq \eps d^2$ for some $\eps \leq 1/4$ and 
every pair of vertices has $\frac{d^2}{n} (1 \pm \eps)$ common neighbours. Then 
\begin{align*}
\Dfrac{4\,\Exp{u(\bm{X})}}{\lambda^2(1-\lambda)^2} &= \Dfrac{n}{4d}\cdot \Dfrac{1 - 6\drat+ 6\drat^2}{\drat(1 - \drat)}\left( 1 + O\left( \Dfrac{1}{d}\right) \right),\\ 
\Dfrac{4\Exp{v^2(\bm{X})}}{\lambda^3(1-\lambda)^3} &= \Dfrac{n}{3d}\cdot \Dfrac{1 - 4\drat+4\drat^2}{\drat(1-\drat)} \left( 1 + O\left(\Dfrac{1}{d}+\eps \right) \right).
\end{align*}
\end{lemma}
\begin{proof}
We follow the method of  \cite[Lemma 9]{GIM2022}. Define
\begin{align*}
\sigma_{jk, \ell m } &= \CovA{X_j + X_k, X_{\ell} + X_m} \\
 &=
 \CovA{X_j,X_\ell}+\CovA{X_j,X_m}+\CovA{X_k,X_\ell}+\CovA{X_k,X_m}.
\end{align*}
The value of $\CovA{X_j, X_k}$ is equal to the corresponding entry of $\frac12 Q^{-1}$. Define 
\[\hbar := \Dfrac{1}{2d}\left(1 + \Dfrac{1}{d} - \Dfrac{1}{n}\right), 
\qquad e(jk,\ell m):=\Abs{\{j\ell, jm, k\ell,km\} \cap G}.
\] 
From Lemma \ref{lem:qinv}, we obtain
\begin{align}\label{eqn:sigma}
\sigma_{jk, \ell m } = 
\begin{cases}
	2 \hbar \left( 1 + O\left(\frac{1}{d} \right) \right), & \text{if }|\{j,k\} \cap \{\ell, m\}| = 2,\\
	\hbar \left( 1 + O\left(\frac{1}{d} \right) \right), & \text{if } |\{j,k\} \cap \{\ell, m\}| = 1,\\
	\hbar\left( \frac{2}{n} - \frac{e(jk,\ell m)}{d} \pm \frac{16\eps}{n} \right), & \text{otherwise.} 
\end{cases}
\end{align}
From  Theorem \ref{T:isserlis} we find that, for all $(j,k,\ell,m)$,
\begin{align*}
\Exp{(X_j + X_k)^4} &= 3\sigma_{jk,jk}^2,\\
\Exp{(X_j + X_k)^3(X_\ell + X_m)^3} &= 9\sigma_{jk,jk}\sigma_{\ell m,\ell m}\sigma_{jk,\ell m} + 6\sigma_{jk,\ell m}^3.
\end{align*}
Recalling the definition of $u$ from \eqref{def:uv}, we find that 
\begin{align*}
\Dfrac{4\,\Exp{u(\bm{X})}}{\lambda^2(1-\lambda)^2} &=    \Dfrac{1 - 6\drat + 6\drat^2}{2\lambda(1-\lambda)} \sum_{jk \in G} \sigma^2_{jk,jk} 
= \Dfrac{1 - 6\drat + 6\drat^2}{\lambda(1-\lambda) 
	nd\hbar^2
}\left(1+O\left(\Dfrac1d\right)\right)  
\\&= \Dfrac{n}{4d}\cdot 
\Dfrac{1 - 6\drat+ 6\drat^2}{\drat(1 - \drat)} \left( 1 + O\left( \Dfrac{1}{d} \right) \right).
\end{align*}
Similarly, using Isserlis' formula, we get that
\begin{align}\label{eqn:varisserli}
\Dfrac{4\Exp{v^2(\bm{X})}}{\lambda^3(1-\lambda)^3} = \Dfrac{(1-2\lambda)^2}{3\lambda(1-\lambda)}
\sum_{jk,\ell m \in G} \left(  3\sigma_{jk,jk}\sigma_{\ell m, \ell m} \sigma_{jk, \ell m} + 2\sigma_{jk, \ell m}^3 \right).
\end{align}
To compute the sum in \eqref{eqn:varisserli}, we partition it based on the value of $|\{j,k\} \cap \{\ell, m\}|$. 

First, there are exactly $\frac{1}{2}nd$ choices for an ordered pair of edges $(jk, \ell m)$ such that $jk = \ell m$, and for each such term, by \eqref{eqn:sigma}, the corresponding summand is 
\[
3\sigma_{jk,jk}\sigma_{\ell m, \ell m} \sigma_{jk, \ell m} + 2 \sigma_{jk, \ell m}^3 = 40 \hbar^3 \left( 1 + O\left(\Dfrac{1}{d} \right) \right) = 
O\left( \Dfrac{1}{d^3} \right).
\]
Thus, the total contribution of such terms to the sum in \eqref{eqn:varisserli} is $O\left( n/d^2 \right)$. 

Second, there are exactly $nd(d-1)$ choices for $(jk, \ell m)$ such that they overlap on exactly one vertex (without loss of generality, we can suppose $j=\ell$). For each such choice of $\{j,k,\ell,m\}$, by \eqref{eqn:sigma}, we find that the value of the corresponding summand is 
\begin{align*}
3 \sigma_{jk,jk}\sigma_{\ell m, \ell m} \sigma_{jk, \ell m} + 2\sigma_{jk, \ell m}^3 =  14 \hbar^3 \left( 1 + O\left( \Dfrac{1}{d} \right) \right).
\end{align*}
Thus, the total contribution of pairs $(jk,\ell m)$ where $jk$ and $\ell m$ intersect on exactly one vertex
to the sum in \eqref{eqn:varisserli}
equals $14 \hbar^3 nd^2\left( 1 + O\left( \dfrac{1}{d} \right) \right)= \dfrac{7 n}{4d}\left( 1 + O\left( \dfrac{1}{d} \right)\right)$.

Finally, we focus on the choices for $(jk,\ell m)$ such that $|\{j,k\} \cap \{\ell, m\}| = 0$. There are $\frac{nd}{2}\left( \frac{nd}{2} - 2d + 1\right)$ choices for $\{j,k,\ell,m\}$ in this case. Let $\mathcal{S}_0$ be the set of these choices for $(jk, \ell m)$. Using \eqref{eqn:sigma}
and the assumption that $n \leq \eps^2 d$, we find that
\begin{align*}
\sum_{(jk,\ell m) \in\mathcal{S}_0} \Bigl( 3 
&\sigma_{jk,jk}\sigma_{\ell m, \ell m} \sigma_{jk, \ell m} + 2\sigma_{jk, \ell m}^3 \Bigr) 
\\[-1ex] &= \sum_{(jk,\ell m) \in\mathcal{S}_0} 12 \hbar^2
\left( 1 + O\left(\Dfrac{1}{d}  \right)\right) \sigma_{jk,\ell m} \\
&= \sum_{(jk,\ell m) \in\mathcal{S}_0} \Dfrac{3}{2d^3}
\left( 1 + O\left(\Dfrac{1}{d}  \right)\right)
\left( \Dfrac{2}{n} - \Dfrac{e(jk,\ell m)}{d} \pm \Dfrac{16\eps}{n} \right)
\\
&= \Dfrac{3n}{4d} + O\left(\Dfrac{n}{d^2} + \Dfrac{\eps n}{d}\right) - \Dfrac{3}{2d^4}
\left( 1 + O\left(\Dfrac{1}{d}\right)\right)\sum_{(jk,\ell m) \in\mathcal{S}_0} e(jk,\ell m).
\end{align*}
Note that $\sum_{(jk,\ell m) \in\mathcal{S}_0} e(jk,\ell m)$ is exactly the number of oriented $3$-paths in $G$ which are not closed, that is, do not form a triangle.  The total number of oriented $3$-paths is $nd(d-1)^2$.
By the assumption for the number of common neighbours, we get that the number of triangles is $O(d^3)$. Therefore,
\[
\sum_{(jk,\ell m) \in\mathcal{S}_0} e(jk,\ell m) = nd^3 \left( 1 + O\left( \Dfrac{1}{d} \right) \right).
\]
Thus, the total contribution of pairs $(jk,\ell m)$ where $jk$ and $\ell m$ are not adjacent
to the sum in \eqref{eqn:varisserli}
equals 
$-\dfrac{3n}{4d} + O\left(\dfrac{n}{d^2} + \dfrac{\eps n}{d}\right)$.

Combining the estimate above, we estimate the sum in \eqref{eqn:varisserli} to be
\[
\Dfrac{4\Exp{v^2(\bm{X})}}{\lambda^3(1-\lambda)^3} = \Dfrac{(1-2\lambda)^2}{3\lambda(1-\lambda)} \left(\Dfrac{n}{d} + O\left(\Dfrac{n}{d^2} + \Dfrac{\eps n}{d}\right)\right),
\]
which completes the proof.
\end{proof}

\subsection{Proof of Theorem \ref{T:main2}
and Theorem \ref{T:main1}(ii)}\label{S:main2}

For Theorem~\ref{T:main1}(ii) and 
Theorem~\ref{T:main2}, we will also need the following lemma that ensures that the number of common neighbours of every pair of vertices in $\calG(n,d)$
is concentrated around $d^2/n$, so  we can use Corollary \ref{cor-jmb}(b) for a typical $d$-regular graph.

\begin{lemma}\label{lem:MMcommon}
If $\min\{d,n-d\} = \omega(n/\log n)$ 
then, with probability $1-e^{-\omega (n/\log^6 n)}$, all pairs of distinct vertices of $\G_d \sim \calG(n,d)$ have $\frac{d^2}{n} \pm \frac{d^3}{n^2\log n}$ common neighbors.
\end{lemma}	
\begin{proof}
If $d\ge \bigl(1-\frac{3}{4\log n}\bigr)n$ then 
the statement holds trivially since the number of common neighbours always lies in $[2d-n,d]$,
which is a subset of the desired interval if $n$ is large enough.
In the following, we assume that $\frac{n}{\log n}\ll d\le\bigl(1-\frac{3}{4\log n}\bigr)n$. 

Our proof of Lemma \ref{lem:MMcommon} strengthens the arguments of
\cite[Theorem 2.1 (case 2)]{KSVW2001}. 
Let $\mathcal{C}_k$ denote the set of $d$-regular graphs on $[n]$ such that two fixed vertices $u,v$ have exactly $k$ common neighbours.
In \cite[Section 3]{KSVW2001}, it is shown that
\[
  |\mathcal{C}_k| \sim 
  \frac{(n-2)!\,\sqrt{2} e^{1/4} \bigl(\lambda'^{\lambda'}(1-\lambda')^{1-\lambda'}\bigr)^{\binom{n-2}{2}}}
  {k!\,(d-k)!^2(n-2d+k)!}
  \binom{n-3}{d}^{n-2d+k-2}
  \binom{n-3}{d-1}^{2d-2k}
  \binom{n-3}{d-2}^k,
\]
where $\lambda'=\frac{d(n-4)}{(n-2)(n-3)}$.
Dividing two values, we have for $\frac{d^2}{n}+1 < k\le k+j\le d$ that
\begin{align}
   \frac{|\mathcal{C}_{k+j}|}{|\mathcal{C}_{k}|} & \sim
   \biggl(1 + \frac{n-2}{d(n-d-1)}\biggr)^{\!j}
   \prod_{q=k+1}^{k+j} \frac{(d-q+1)^2}{q(n-2d+q-2)} \notag\\
   &\le \biggl(1+\frac{n-2}{d(n-d-1)}\biggr)^{\!j}
     \biggl( \frac{(d-k)^2}{(k+1)(n-2d+k-1)}\biggr)^{\!j}. \label{eq:Crat}
\end{align}            
Note that~\eqref{eq:Crat} is decreasing in $k$ for fixed~$j$.
Define $j_0:=\bigl\lceil\frac{d^3}{2n^2\log n}\bigr\rceil$ and
$k_0 := \bigl\lfloor\frac{d^2}{n}+j_0\bigr\rfloor$. 
The upper bound on $d$ ensures that $k_0<d$.
Using the assumed bounds on $d$, we have
\[
   \frac{(d-k_0)^2}{(k_0+1)(n-2d+k_0-1)} 
    = \frac{1}{1 + \frac{k_0 n-d^2-2d+n-1}{(d-k_0)^2}} 
    \le \frac{1}{1+\log^{-2} n},
\]
where for the final step we used $k_0\ge \frac{d^2}{n}+\frac{d^3}{2n^2\log n}-1$.
Consequently, 
\begin{align*}
   \frac{\sum_{j\ge j_0}|\mathcal{C}_{k_0+j}|}
        {\sum_{j\ge 0}|\mathcal{C}_j|}
   &\le
   \frac{\sum_{j\ge j_0}|\mathcal{C}_{k_0+j}|}
        {\sum_{j\ge j_0}|\mathcal{C}_{k_0-j_0+j}|} 
    \le \max_{j\ge j_0}\frac{|\mathcal{C}_{k_0+j}|}{|\mathcal{C}_{k_0-j_0+j}|}
   \\ &= 
   O(1) \biggl(\frac{1}{1+\log^{-2} n}\biggr)^{j_0}
    = O\bigl(e^{-\frac{d^3}{n^2 \log^3 n}}\bigr) = e^{-\omega (n/\log^6 n)}.
\end{align*}

A parallel argument shows that with probability $1-e^{-\omega (n/\log^6 n)}$,
the number of common neighbours of $u$, $v$ is at least 
$\frac{d^2}{n} - \frac{d^3}{n^2\log n}$.
Applying the union bound over $u,v$, we get that the same bounds hold for all pairs of distinct vertices.
\end{proof}

Now we are ready to prove 
Theorem \ref{T:main2}
and Theorem \ref{T:main1}(ii). Let 
\begin{equation}\label{def:Rhat}
\widehat{R}_h(d):= 2^{1/2} e^{1/4}\,
\Dfrac{h^{\frac{ hn}{2}} (d-h)^{\frac{ (d-h)n}{2}}}{d^{\frac{ dn}{2}}}
\binom{d}{h}^{n}.
\end{equation}
The claim of Theorem \ref{T:main2} is equivalent to 
\begin{equation}\label{claim:main2}
\E\left|\calR_{d_1} (\G_{d_1+d_2})\right|
= (1+o(1))\widehat{R}_{d_1}(d_1+d_2). 
\end{equation}
Using Stirling's approximation, observe that 
\begin{align*}
\binom{d_1+d_2}{d_1}  &=  \Dfrac{\sqrt{2\pi (d_1+d_2)}\left(\dfrac{d_1+d_2}{e}\right)^{d_1+d_2} e^{\frac{1}{12(d_1+d_2)} + o(\frac1n)}}{
2\pi \sqrt{d_1d_2}\left(\dfrac{d_1}{e}\right)^{d_1} 
\left(\dfrac{d_2}{e}\right)^{d_2}  e^{\frac{1}{12d_1}+\frac{1}{12d_2}+o(\frac1n)}} \notag
\\
 &= \Dfrac{\left(\lambda^{\lambda }  (1 - \drat)^{1-\lambda}\right)^{-(d_1+d_2)}}{\sqrt{2\pi d_1d_2/(d_1+d_2)}} 
 \exp\left(  \Dfrac{1}{12(d_1+d_2)} \left(1- \Dfrac{1}{\lambda(1-\lambda)}\right) +o\left(\Dfrac1n\right)\right).
\end{align*}
By the assumption that $d_1, d_2 = \omega(n/\log n)$, we have
\[
\Dfrac{d_1 d_2 }{d_1 + d_2} = \omega\Bigl(\Dfrac{n}{\log n}\Bigr).
\]
Then, combining Corollary \ref{cor-jmb}(b)
and Lemma \ref{lem:MMcommon} with $d=d_1+d_2$ and $h=d_1$, we find that, with probability $1-e^{-\omega(n/\log^6 n)}$,
\begin{align}
|\calR_{d_1} (\G_{d_1+d_2})|&=
\Dfrac{   \left(\lambda^{\lambda }  (1 - \drat)^{1-\lambda}\right)^{-\frac{(d_1+d_2)n}{2}} }
{(\pi (d_1+d_2) \drat(1-\drat))^{\frac{n}{2}}    } \,2^{\frac12}\exp\left(  \Dfrac{n}{12(d_1+d_2)}
\left(1- \Dfrac{1}{\lambda(1-\lambda)}\right)+\dfrac{1}{4}   + o(1)\right) \notag
\\
&=  (1+o(1))\widehat{R}_{d_1}(d_1+d_2). \label{R-con}
\end{align}

First, we prove \eqref{claim:main2} under the additional assumption that $d_1+d_2 \geq \frac23(n-1)$.
Using  Corollary \ref{cor-jmb}(a), we get that, for an arbitrary $(d_1+d_2)$-regular graph $G$, 
\[
|\calR_{d_1} (G)| =  \widehat{R}_{d_1}(d_1+d_2) \,  e^{o(\log n)}.
\]
Using also   \eqref{R-con}, we get  that
\[
    \E\left|\calR_{d_1} (\G_{d_1+d_2})\right| = 
    (1+o(1)) \widehat{R}_{d_1}(d_1+d_2)  
 (1-\Pr(\mathcal{E}))
    + \widehat{R}_{d_1}(d_1+d_2) \,  e^{o(\log n)} \Pr(\mathcal{E}),
\]
for some event $\mathcal{E}$ of probability at most $e^{-\omega(n/\log^6 n)}$.
Claim \eqref{claim:main2} follows.

Now assume that $d_1+d_2 \leq \frac23(n-1)$.
We can also assume that $d_1\leq d_2$ since, 
\[
|\calR_{d_1}(\G_{d_1+d_2})| = |\calR_{d_2}(\G_{d_1+d_2})|.
\]
Then, we get $d_2 + d_3 \geq \frac23(n-1)$,
where $d_3:= n-1-d_1-d_2$. Note also 
\[
d_3 \geq \dfrac13(n-1) = \omega\Bigl(\Dfrac{n}{\log n}\Bigr), \qquad 
n-d_3 \geq d_2 = \omega\Bigl(\Dfrac{n}{\log n}\Bigr).
\]
We have already established that 
\begin{equation}\label{eq:al}
\E|\calR_{d_2}(\G_{d_3+d_2})| =(1+o(1))\widehat{R}_{d_2}(d_2+d_3)
\end{equation}
with probability tending to 1.
By a simple double-counting argument, we find that
\[
|\calR_{d_1+d_2}(K_n)| \,\E|\calR_{d_1}(\G_{d_1+d_2})|
= |\calR_{d_2+d_3}(K_n)| \,\E|\calR_{d_2} (\G_{d_3+d_2})|, 
\]
which equals the number
of partitions of the complete graph $K_n$ into three disjoint regular graphs with degrees
$d_1$, $d_2$, and $d_3$. Finally, applying \eqref{eq:reg}, we get that 
\begin{align*}
\Dfrac{|\calR_{d_2+d_3}(K_n)|}
{|\calR_{d_1+d_2}(K_n)| }
= (1+o(1)) \Dfrac{(d_2+d_3)^{\frac{d_2+d_3}{2}}
d_1^{\frac{d_1}{2}} \binom{n-1}{d_2+d_3 }^n}{
(d_1+d_2)^{\frac{d_1+d_2}{2}}
d_3^{\frac{d_3}{2} }\binom{n-1}{d_1+d_2}^n }
= (1+o(1))\Dfrac{\widehat{R}_{d_1}(d_1+d_2)}{\widehat{R}_{d_2}(d_2+d_3)}.
\end{align*}
This together with \eqref{eq:al} implies \eqref{claim:main2} and completes the proof of Theorem \ref{T:main2}.

For Theorem \ref{T:main1}(ii), combining \eqref{claim:main2} and \eqref{R-con}, we get that 
\[
|\calR_{d_1}(\G_{d_1+d_2})| = (1+o(1))\E |\calR_{d_1}(\G_{d_1+d_2})| 
\]
with probability tending to 1.
Using Lemma \ref{l:reduction} (the Reduction lemma), we establish Conjecture 
\ref{con11}.  We obtain   Conjecture 
\ref{con1}.
by swapping the roles of $d_1$, $d_2$, $d_3$, similarly to Section \ref{S:Thm_i}.

\section{Sprinkling dense with sparse}\label{S:dense-sparse}

In this section, we prove Theorem \ref{T:main1}(iii). As explained in 
Section \ref{S:Thm_i}, we need to prove \eqref{eq:first}, \eqref{eq:second}, and \eqref{eq:third} and then apply Lemma \ref{l:reduction} (the Reduction lemma).

The assumption $d_1^2\left(\dfrac{n}{\min\{d_2,n-d_2\}}\right)^{2d_1}\!\!\leq\frac{1}{3}
\log n$ implies that 
\begin{equation}\label{ass_impl}
d_1 = O\bigl(\sqrt{\log n}\,\bigr), \qquad 
\min\{d_2,n-d_2\} = \Omega\Bigl(\Dfrac{n}{\sqrt{\log n}}\Bigr).
\end{equation}
Repeating the arguments of Section \ref{S:main2}, namely, using Corollary \ref{cor-jmb} to derive formulas similar to \eqref{claim:main2} and \eqref{R-con}, we get that 
\[
|\calR_{d_2}(K_n\setminus \G_{d_1})| = (1+o(1))\E |\calR_{d_2}(K_n \setminus \G_{d_1})| ,
\]
with probability $1 - e^{-\omega(n/\log^6 n)}$,
which implies  \eqref{eq:second}.

Due to \eqref{ass_impl}, we can always swap roles of $d_2$ and $d_3=n-1-d_1-d_2$. Thus, we only need to focus on proving \eqref{eq:third}.
Let 
$
d := d_1 +d_2.
$ 
It is sufficient to show that 
\begin{equation}\label{ii-sufficient}
\E|\mathcal{R}_{d_1}(\G_d)|^2 \leq (1+o(1))\bigl(\E|\mathcal{R}_{d_1}(\G_d)|\bigr)^2. 
\end{equation}
Indeed, by Chebyshev's inequality, \eqref{ii-sufficient} implies that 
\[
|\mathcal{R}_{d_1}(\G_d)| = (1+o(1))\E|\mathcal{R}_{d_1}(\G_d)| ,
\]
so Conjecture \ref{con1}(a) follows by Lemma~\ref{l:reduction}. 

To derive \eqref{ii-sufficient}, we write 
\begin{equation}\label{eq:E-P}
\begin{aligned}
\E|\mathcal{R}_{d_1}(\G_d)| &= \sum_{H_1  \in \calR_{d_1}(K_n)} P_d(H_1),
\\
\E|\mathcal{R}_{d_1}(\G_d)|^2 &= \sum_{H_1,H_2  \in \calR_{d_1}(K_n)} P_d(H_1 \cup H_2),
\end{aligned}
\end{equation}
where $P_d(H)$ denotes the probability that $\G_d \sim \calG(n,d)$
contains $H$ as a subgraph. 
To estimate $P_d(H)$, we use the following result which is a special case of \cite[Corollary 2.2]{ranx} when the degrees of $H$ are small.
\begin{theorem}[McKay, \cite{ranx}] 
Let $\eps>0$ be small enough. Let $\tilde H$ be a graph with degree sequence $h_i$, $i\in[n]$, and maximum degree at most $n^{\eps}$.  Set $\mu=\frac{1}{n}\sum_{i=1}^n h_i^2$, $m=\frac{1}{2}\sum_{i=1}^n h_i$. Assume $\min\{d,n-d-1\}\ge \dfrac{n}{\log n}$. Then
$$
P_d(\tilde H)=
\Bigl(\Dfrac{d}{n-1}\Bigr)^m\exp\Bigl(\Dfrac{n-1-d}{4d}\Bigl(\Dfrac{4m^2}{n^2}+\Dfrac{4m}{n}-2\mu\Bigr)+o(n^{-1/4})\Bigr).
$$
In particular, if $H$ is $h$-regular then 
$$
P_d(H)=
\Bigl(\Dfrac{d}{n-1}\Bigr)^{hn/2}\exp\Bigl(-\Dfrac{n-1-d}{4d}h(h-2)+o(n^{-1/4})\Bigr).
$$
\label{th:GIM}
\end{theorem}

To estimate
the sum in $\E|\mathcal{R}_{d_1}(\G_d)|^2$, we first obtain the distribution of the number of common edges of $H_1$ and a random relabelling of $H_2$.
Then, we have completed the proof of Theorem \ref{T:main1}(iii) in Section \ref{S-main-2}.

\subsection{The common edges of two regular graphs}\label{S:common-sparse}

Throughout this section $H_1$ and $H_2$ are two $h$-regular graphs on the same vertex set~$[n]$.
Let $H_2^*$ be a randomly relabelled copy of $H_2$.

\begin{lemma}\label{overlap}
Assume $1\le h=o(n^{1/4})$.
Then, for $0\le m\le \max\{8h^2,\log n\}$, the probability that $H_1$ and $H_2^*$ have exactly $m$ common edges is
\[
  \mathbb{P}\bigl(|H_1\cap H_2^*|=m\bigr)=\Dfrac{e^{-h^2/2}}{m!} 
  \left(\Dfrac{h^2}{2}\right)^m\left(1 + O\Bigl(\Dfrac{h^4+\log^2 n}{n}\Bigr)\right).
\]
\end{lemma}
\begin{proof}
This is proved in \cite{HM1} but not explicitly stated there.
It is stated in~\cite[Lemma~2.1]{HM1} that
\[
   \mathbb{P}\bigl(|H_1\cap H_2^*|>\max\{8h^2,\log n\}\bigr) = O\Bigl(\Dfrac{h^4}{n}\Bigr),
\]
while~\cite[Theorem 2.5]{HM1} says
\[
    \mathbb{P}\bigl(|H_1\cap H_2^*|=0\bigr) = e^{-h^2/2}\Bigl(1+O\Bigl(\Dfrac{h^4}{n}\Bigr)\Bigr),
\]
and~\cite[Lemma 2.2]{HM1} says, for $1\le m\le \max\{8h^2,\log n\}$, that
\[
  \frac{\mathbb{P}\bigl(|H_1\cap H_2^*|=m\bigr)}{\mathbb{P}\bigl(|H_1\cap H_2^*|=m-1\bigr)}
  = \Dfrac{h^2}{2m}\Bigl(1+O\Bigl(\Dfrac{h^2}{n}+\Dfrac mn\Bigr)\Bigr).
\]
The lemma now follows from combining these three bounds.
\end{proof}
The intersection of $H_1$ and $H_2^*$ is typically small, as shown by the following lemma.


\begin{lemma}\label{lem:common}
Let
\[ m_{\alpha} := \Dfrac{\alpha h^2 n}{2(n-h-1)}, \qquad\text{where $\alpha\ge e$}.\]
Then the probability that $H_2^*$ has $m_\alpha$ or more
edges in common with~$H_1$ is
\[
\mathbb{P}(|H_1\cap H_2^*|\geq m_{\alpha})\leq 2\,\left( \Dfrac e\alpha\right)
^{\! \textstyle\bigl\lfloor\frac{(\alpha-1)hn}{4(n-h-1)}\bigr\rfloor}.
\]
\end{lemma}

\begin{proof}
Define a directed graph $G$ with vertices $v_0,v_1,\ldots\,$, with
$v_i$ representing the set $\mathcal{S}_i$ of labellings of $H_2$ that have exactly
$i$ edges in common with~$H_1$.
For a permutation $\sigma\in S_n$, let $H^\sigma_2$ denote the isomorphic copy 
of $H_2$ in which each vertex~$v$ has been relabelled in accordance with the permutation $\sigma$. 

Now we define a switching on the multiset of all $n!$ labellings of~$H_2$.
Assume $H^\sigma_2\in\mathcal{S}_i$. Choose an edge
$uw$ in common between $H_1$ and $H^\sigma_2$, choose one end, say~$u$,
choose a vertex~$z$ which is neither~$u$ nor adjacent to $u$ in $H_2^\sigma$
and apply the
transposition $(u\,z)$ to $H^\sigma_2$. Now $uw$ is no longer a common edge.
Note that other common edges can be created or destroyed, but that
doesn't matter: all we need to notice is that at most $2h$ common
edges can be lost (the edges incident with~$u$ or~$z$ in~$H_2^\sigma$).
If $H_1$ and $H_2^\sigma$ have $i$ common edges, this operation can
be performed in exactly $ia$ ways, where 
\[ 
a:=2(n-h-1).
\]

For the reverse switching, choose an edge $uw$ of $H_1$ that is not in $H_2^\sigma$,
choose one end,
say~$u$, choose a neighbour $z$ of~$w$ in $H_2^\sigma$,
and apply the transposition $(u\,z)$ to~$H_2^\sigma$. 
Now $uw$ is a common edge.  This can be done
in at most 
\[
b := h^2n
\] 
ways (we only need an upper bound this time). 
As before, other common edges can be created or lost but we don't care.

The digraph $G$ has an edge $v_iv_j$ whenever there is a switching that takes
a labelling in $\mathcal{S}_i$ to one in~$\mathcal{S}_j$. Loops are
possible, but $v_0$ is the unique sink.
We shall apply~\cite[Corollary~1]{HM2010}. For convenience, let us state it here:	
\begin{lemma}[{\cite{HM2010}, Corollary 1}]\label{lem:cor1}
Let $G$ be a directed graph with vertices $v_0,v_1,\ldots$
such that $v_0$ is the unique sink.
Let $C:V(G)\to\mathbb{R}_+$ be a non-negative function such that
$\sum_{i\geq 0} C(v_i) = 1$. Suppose that each edge $v_iv_j$ has
$j-i\ge -K$ for some integer $K>0$. 
Suppose that some real quantities $\{ x(v_iv_j) \mid v_iv_j\in E(G)\}$ satisfy
the following inequalities for some $\rho>0:$
\begin{align*}
	x(v_iv_j) &\ge 0, &&\text{for all $v_iv_j\in E(G)$}; \\[0,5ex]
	\sum_{j\st v_iv_j\in G} \frac{\rho}{i} x(v_iv_j)&\ge C(v_i),
	&&\text{for all $i\ge 1$}; \\[-0.5ex]
	\sum_{j\st v_jv_i\in G} x(v_jv_i) &\le C(v_i), &&\text{for all $i$}.
\end{align*}
Then for any integer $m>\max\{\rho,K-1\}$,
\[ \sum_{i\ge m} C(v_i) \le \Dfrac{1}{1-\rho/m} (e\rho/m)^k, \]
where $k=\lfloor (m + \min\{0,K-\rho-1\})/K\rfloor$.
\end{lemma}

To apply Lemma~\ref{lem:cor1}, we can make the following interpretations.
Let $s(v_iv_j)$ be the total number of switchings that start in $\mathcal{S}_i$ and
end in~$\mathcal{S}_j$.
By our counting, we have
\begin{align*}
\sum_{j:v_iv_j\in G} s(v_iv_j) &\ge ia \,|\mathcal{S}_i|, &&\text{for all $i\ge 1$;} \\[-0.5ex]
\sum_{j:v_jv_i\in G} s(v_jv_i) &\le b \,|\mathcal{S}_i|, &&\text{for all $i.$}
\end{align*}
Now, for all $i,j$, substitute $C(v_i):=|\mathcal{S}_i|/n!$,
$x(v_iv_j):=s(v_iv_j) / (b n!)$, and
\[ 
\rho := \Dfrac ba = \Dfrac{h^2 n}{2(n-h-1)}.
\]
Together with $K=2d_1$, this gives the conditions of Lemma~\ref{lem:cor1}.

We have $m_{\alpha}>\max\{\rho,K-1\}$. Also
\[ 
k\geq\biggl\lfloor\Dfrac{m_{\alpha}-\rho}{K}\biggr\rfloor
= \biggl\lfloor\Dfrac{(\alpha-1)hn}{4(n-h-1)}\biggr\rfloor,
\quad 1-\Dfrac{\rho}{m_{\alpha}} >\dfrac12
\quad\text{and}\quad  \Dfrac{e\rho}{m_{\alpha}}=\Dfrac{e}{\alpha}.
\]
Lemma~\ref{lem:common} follows.
\end{proof}

\subsection{Completing the proof of Theorem \ref{T:main1}(iii)}\label{S-main-2}
For Theorem \ref{T:main1}(iii), we use the results of Section \ref{S:common-sparse} with $h=d_1$. In particular,
recall that $H_2^*$ is a randomly relabelled copy of $H_2$.
Starting from \eqref{eq:E-P}, we get that
\begin{align}
\Dfrac{\mathbb{E} \,|\mathcal{R}_{d_1}(\G_d)|^2}
{\bigl(\mathbb{E}\, |\mathcal{R}_{d_1}(\G_d)|\bigr)^2}
&= \Dfrac{\sum_{H_1,H_2\in\mathcal{R}_{d_1}(K_n)}  P_d(H_1\cup H_2)}
{\bigl(\sum_{H\in\mathcal{R}_{d_1}(K_n)} P_d(H)\bigr)^2} \notag
\\
&=
\Dfrac{\sum_{H_1,H_2\in\mathcal{R}_{d_1}(K_n)} \E P_d(H_1\cup H_2^*)}
{\sum_{H_1,H_2\in\mathcal{R}_{d_1}(K_n)}  P_d(H_1)P_d(H_2)}. \label{eq:starting}
\end{align}

Let $H_1\neq H_2$ 
be two $d_1$-regular graphs on $[n]$. Let $m_i$ be the number of edges adjacent to the vertex 
$i$ in $H_1\cap H_2$, and $m=\frac{1}{2}\sum_i m_i$ be the total number of edges in $H_1\cap H_2$.
From Theorem~\ref{th:GIM} (its assumptions hold due to \eqref{ass_impl}), we get that
\begin{align*}
P_d(H_1\cup H_2)&=
\Bigl(\Dfrac{d}{n-1}\Bigr)^{d_1n-m} \\
&\times
\exp\biggl(\Dfrac{n-1-d}{4d}
\biggl(-4d_1^2+2\left(2d_1-\Dfrac{2m}{n}\right)+\Dfrac{4m^2}{n^2}
-2\sum_{i=1}^n\Dfrac{m_i^2}{n}\biggr)+o(n^{-1/4})\biggr).
\end{align*} 
Using the second part of Theorem~\ref{th:GIM},  we estimate 
\[
\Dfrac{P_d(H_1\cup H_2)}{P_d(H_1)P_d(H_2)} \leq 
\left(\Dfrac{n-1}{d}\right)^{|H_1\cap H_2|}
\exp\Bigl(-\Dfrac{(n-1-d)d_1^2}{2d} +o(n^{-1/4})\Bigr).
\]
Therefore, combining the above bound and \eqref{eq:starting}, we get that
\begin{align*}
&\Dfrac{\mathbb{E}\, |\mathcal{R}_{d_1}(\G_d)|^2}
{\left(\mathbb{E}\, |\mathcal{R}_{d_1}(\G_d)|\right)^2}
=  \frac{\sum_{H_1,H_2\in\mathcal{R}_{d_1}(K_n)}\sum_{m=0}^{nd_1/2} \Pr(H_1\cup H_2^* \mid H_1 \cap H_2^* =m) \, \Pr (H_1 \cap H_2^* =m)}{\sum_{H_1,H_2\in\mathcal{R}_{d_1}(K_n)}  P_d(H_1)P_d(H_2)}
\\
&\qquad \leq 
\exp\left(-\Dfrac{(n-1-d)d_1^2}{2d} +o(n^{-1/4}) \right)
\sum_{m =0}^{nd_1/2}  \left(\Dfrac{n-1}{d}\right)^{m} \max_{H_1,H_2 \in \calR_{d_1}(K_n)} \Pr(|H_1 \cap H_2^*|=m).
\end{align*}
Using 
Lemma~\ref{overlap}, the contribution of the terms with $m< \log n$ to the sum above is bounded by
\[
(1+o(1))
\sum_{m< \log n }\Dfrac{e^{-d_1^2/2}}{m!} \Bigl(\Dfrac{d_1^2}{2}\Bigr)^m\Bigl(\Dfrac{n-1}{d}\Bigr)^m
\leq  (1+o(1))
\exp\left( \Dfrac{(n-1)d_1^2}{2d} - \Dfrac{d_1^2}{2}  \right). 
\]
Using Lemma~\ref{lem:common}
and the assumption that 
$d_1^2 \left(\dfrac{n}{d_2}\right)^{2d_1}\leq \frac13 \log n$, the contribution of the remaining terms with $m \geq \log n$ 
is bounded by 
\begin{align*}
\sum_{m\geq \log n }\biggl( \Dfrac {(e+o(1)) d_1^2}{m}\biggr)^{\frac{m}{2d_1}}
\left(\Dfrac{n-1}{d}\right)^m &=
\sum_{m \geq \log n} \biggl( \Dfrac {e+o(1) }{3}\biggr)^{\frac{m}{2d_1}} 
\left(\Dfrac{3d_1^2}{\log n}\right)^{\frac{m}{2d_1}}\left(\Dfrac{n-1}{d}\right)^m 
\\
&\leq \sum_{m \geq \log n} \biggl( \Dfrac {e+o(1)}{3}\biggr)^{\frac{m}{2d_1}}   =o(1).
\end{align*}
To derive the last equality, we 
used $d_1 = O(\sqrt{\log n})$, which is implied by our assumption $d_1^2 \left(\dfrac{n}{d_2}\right)^{2d_1}\leq \frac13 \log n$.
Combining the above, we find that
\begin{align*}
\Dfrac{\mathbb{E}\, |\mathcal{R}_{d_1}(\G_d)|^2}
{\left(\mathbb{E}\, |\mathcal{R}_{d_1}(\G_d)|\right)^2} &\leq   \exp\left(-\Dfrac{(n-1-d)d_1^2}{2d} +o(n^{-1/4}) \right) 
\left( (1+o(1))
\exp\Bigl( \Dfrac{(n-1)d_1^2}{2d} - \Dfrac{d_1^2}{2}  \Bigr) +o(1)\right)
\\&=1+o(1).
\end{align*}
Thus we have established \eqref{ii-sufficient}, which completes the proof.

\nicebreak

\end{document}